\documentclass{amsart}

\usepackage{amsmath,amssymb,amsthm} %,mathabx}
\usepackage{longtable}
\usepackage{epsf,stmaryrd,amscd} 
\usepackage{xspace}
\usepackage{graphicx,epstopdf}
\usepackage{latexsym,mathtools}
\usepackage{bbding} 

\newcommand{\ds}{^\diamondsuit} %alt: \bigstar \dagger

\newcommand{\Sp}{{\mathrm{Span}}}
\newcommand{\Id}{{\mathrm{Id}}}

\newcommand{\SL}{{\mathrm{SL}\,}} 
\newcommand{\PSL}{{\mathrm{PSL}\,}}

\newcommand{\bF}{{\mathbb F}}

\newcommand{\bL}{{\mathbb L}}
\newcommand{\bM}{{\mathbb M}}
\newcommand{\bN}{{\mathbb N}}

\newcommand{\bZ}{{\mathbb Z}}

\newcommand{\fA}{{\mathfrak A}} 
\newcommand{\fB}{{\mathfrak B}} 

\newcommand{\fS}{{\mathfrak S}}

\newcommand{\mBT}{{\mathcal T}{\mathcal B}} 
\newcommand{\mD}{{\mathcal D}} 
\newcommand{\mF}{{\mathcal F}}
\newcommand{\mG}{{\mathcal G}}
\newcommand{\mJ}{{\mathcal J}}

\newcommand{\mP}{{\mathcal P}} 

\newcommand{\mS}{{\mathcal S}}
\newcommand{\mT}{{\mathcal T}} 
\newcommand{\mX}{{\mathcal X}} 
\newcommand{\mY}{{\mathcal Y}}

\newtheorem*{thmh}{Main Theorem}

\newtheorem*{thmT}{Tits Theorem}

\newtheorem{defn}{Definition}[section]
\newtheorem{theorem}[defn]{Theorem}

\newtheorem{lemma}[defn]{Lemma}
\newtheorem{cor}[defn]{Corollary}

\newtheorem{prop}[defn]{Proposition}

\begin{document}
%\begin{frontmatter}
%  \title{Kac-Moody Groups and Completions\footnote{To memory of Kay Magaard, taken too soon}}
%\title{Kac-Moody Groups and Completions\footnote{In memory of Kay Magaard}}
\title{Kac-Moody Groups and Completions}
\author{Inna Capdeboscq}
\email{I.Capdeboscq@warwick.ac.uk}
%\ead{I.Capdeboscq@warwick.ac.uk}
\address{Department of Mathematics, University of Warwick, Coventry, CV4 7AL, UK}
\author{Dmitriy Rumynin}
\email{D.Rumynin@warwick.ac.uk}
%\ead{D.Rumynin@warwick.ac.uk}
\address{Department of Mathematics, University of Warwick, Coventry, CV4 7AL, UK\newline
\hspace*{0.31cm}  Associated member of Laboratory of Algebraic Geometry, National
Research University Higher School of Economics, Russia}
%\thanks{The research was partially supported by the Russian Academic Excellence Project `5--100' and by Leverhulme Foundation.}
\thanks{Both authors were supported by Leverhulme Foundation.
The second author was partially supported by the Russian Academic Excellence Project `5--100'. 
We are grateful to Guy Rousseau for many interesting discussions and his  valuable suggestions.
We would also like to thank the anonymous referees for many helpful recommendations that led to improvement of the paper.}
\date{February 20, 2020}
\subjclass{Primary  20G44; Secondary 22A05}
\keywords{Kac-Moody group, BN-pair, completion}

%\dedicatory{To memory of Kay Magaard, taken too soon}
\dedicatory{In memory of Kay Magaard}
\begin{abstract}
  In this paper
  we construct a new ``pro-$p$-complete'' topological Kac-Moody group
  and compare it to various known topological Kac-Moody groups.
  We come across this group by investigating the process of
  completion of groups with BN-pairs.
  We would like to know whether the completion of such a group
  admits a BN-pair.
  We give explicit criteria for this to happen.
    \end{abstract}

%\begin{keyword}
%  Kac-Moody group
%  \sep
%  BN-pair
%  \sep completion
%
%\MSC[2010]  20G44 \sep 22A05
%\end{keyword}

%\begin{extrainfo}
%To memory of Kay Magaard, taken too soon}
%\end{extrainfo}
%\end{frontmatter}

\maketitle

In this paper we study Kac-Moody groups over finite fields. 

%What is a topological Kac-Moody group?
%To answer this question, one needs to understand
A Kac-Moody group  is a generalisation of the notion of a reductive group to a more general Kac-Moody root datum $\mD$ or a closely related group with a BN-pair. Connected reductive groups are classified via a one-to-one correspondence to root data of finite type. A root datum of finite type yields a group scheme, generalised by 
Tits to a construction of  
a functor $G_\mD$ from the category of commutative
rings to the category of groups \cite{T2,T3}.
For instance, 
the Steinberg central extension $\widehat{\SL_n (\bF_q[z,z^{-1}])}$ is  a Kac-Moody group
$G_\mD(\bF_q)$ for the simply-connected root datum $\mD$ of the affine type $\widetilde{A}_{n-1}$ (see \cite[sec. 6]{CKRu} for further details).
On the other hand, $\SL_n (\bF_q[z,z^{-1}])$ is not of the form $G_\mD(\bF_q)$,
yet it is still called a Kac-Moody group.

A topological Kac-Moody group
is a locally compact totally disconnected topological group 
that contains a Kac-Moody group.
%Often a topological Kac-Moody group is obtained by a completion of $G_\mD(\bF_q)$.
Often it is obtained by completion of a  Kac-Moody group.
In the examples of the previous paragraph, one arrives at topological Kac-Moody groups 
$\widehat{\SL_n (\bF_q((z))\,)}$ and  $\SL_n (\bF_q((z))\,)$.

There are several known 
topological Kac-Moody groups.\footnote{Since this paper a new book by Marquis \cite{Mar3} was published, which would be a comprehensive source for further reading on the subject.}
They are the Mathieu-Rousseau group  $G^{ma+}$, the Carbone-Garland group $G^{c\lambda}$  and  the Caprace-R\'emy-Ronan group $G^{crr}$. 
Each of them contains $G_\mD(\bF_q)$. 
In this paper we show the existence of a new topological Kac-Moody group $\widehat{G}$.

\begin{thmh}
Let $A$ be an irreducible generalised Cartan matrix, $\mD$  a simply connected root datum of type $A$ and $\mathbb{F}=\mathbb{F}_q$  a finite field of characteristic $p$ ($q=p^a$, $a\in\mathbb{N}$). Let $G{\coloneqq}G_\mD(\bF_q)$
be the corresponding minimal Kac-Moody group. Recall that it has a BN-pair $(B,N)$ with $B=U\rtimes T$ where $U=\langle X_{\alpha}\mid \alpha\in\Delta^{re}_+\rangle$ (see Section~2, where the notations are introduced).

There exists a locally compact totally disconnected group $\widehat{G}$ satisfying the following conditions:
\begin{enumerate}
\item $G$ is a dense subgroup of $\widehat{G}$.
\item $\widehat{G}$ has a BN-pair $(\widehat{B}, N)$ where $\widehat{B}=\widehat{U}\rtimes T$ and $\widehat{U}$ is the full pro-$p$ completion of $U$.
\item If $\overline{G}=G^{crr}$ or $C^{c\lambda}$, or in the case when $G$ is dense in $G^{ma+}$ and $\overline{G}=G^{ma+}$, then there exists an open continuous surjective homomorphism $\widehat{G}\twoheadrightarrow \overline{G}$.
\item Let $Z'(\widehat{G}){\coloneqq}Z(G)\times C(\widehat{G})$ where $C(\widehat{G})=\cap_{g\in\widehat{G}} \widehat{U}^g$.  \begin{enumerate}
\item $\widehat{G}/Z'(\widehat{G})$ is a topologically simple group.
\item If $A$ is $2$-spherical, then $\widehat{G}/Z'(\widehat{G})$ is an abstractly simple group.
\end{enumerate}
\end{enumerate}

\end{thmh}

%We would like to study them by investigating
%the process of completion itself.

Let us explain the content of the present paper.
We investigate the process of completion of a group $G$ with a BN-pair
in Chapter 1.
The main result is Theorem~\ref{complete_1}
that contains a sufficient condition for the completion $\widehat{G}$
to inherit a BN-pair from $G$. It relies on Tits' description of groups with BN-pairs
by generators and relations \cite{kn::T}.
The remainder of Chapter 1 contains several technical or user-friendly results about these completions. 
For instance, Theorem~\ref{complete_2} conveniently constructs the completion $\widehat{G}$ together with
its BN-pair.

We put our completion results to good use in Chapter 2.
After quickly recalling the definition of $G_\mD (\bF_q)$
we construct the new group $\widehat{G_\mD (\bF_q)}$.
We compare it to other known completions and address
its topological simplicity in Theorem~\ref{theorem::top_simple}
and
its algebraic simplicity in Theorem~\ref{theorem::alg_simple}.
Thus, the  Main Theorem  is a combination of results in Section 2.

There have been previous attempts to compare various topological Kac-Moody groups:
Capdeboscq and R\'emy \cite{kn::CR},
Baumgartner and R\'emy \cite[2.6]{kn::E},
Marquis \cite{Mar2},
and Rousseau \cite{kn::Rou}
all discuss the maps between different completions at length.
We address these questions in Section 2 
only modulo ``congruence kernel'' $Z^\prime (\widehat{G})$
whose full computation remains mysterious. We devote the last chapter of the paper
to several observations
about $Z'(\widehat{G})$. Our major insight into the nature of the congruence kernel is
its parabolic decomposition in Theorem~\ref{approximate}.

Hristova and Rumynin study representations of topological  Kac-Moody groups \cite{HrRu}. The groups $\widehat{G}$ %constructed here 
are new examples for their theory.

%\section*{Acknowledgement}
%The research was partially supported by the Russian Academic Excellence Project `5--100' and by Leverhulme Foundation
%Both authors were supported by Leverhulme Foundation.
%The second author was partially supported by the Russian Academic Excellence Project `5--100'. 
%We are grateful to Guy Rousseau for many interesting discussions and his  valuable suggestions.
%We would also like to thank the anonymous referees for many helpful recommendations that led to improvement of the paper.

\section{Completion Theorem}
Let $(G,\mT)$ be a Hausdorff topological group ($G$ is a group, $\mT$ is a topology).  
%with a BN-pair $(B,N)$.
The topology determines a right uniformity on $G$ that we now
describe following Bourbaki \cite{kn::Bou}.
Pick a basis $\mJ$ of the topology at $1$. The basis of uniformity $\mT\ds$
is $\mJ\ds=\{ V\ds | V \in \mJ\}$ where
$V\ds = \{ (x,y)\in G^2 %\mbox{ (or } B^2\mbox{) }
| xy^{-1}\in V\}$.
The completion $\widehat{G}$ is the set of all minimal Cauchy filters on $(G,\mT\ds)$.
Recall that a filter is a non-empty collection $\mF$ of open sets closed under intersections and oversets. A filter is Cauchy if it contains arbitrary ``small'' subsets, i.e., for each $V\in\mJ$ there exists $U\in\mF$ such that $xy^{-1}\in V$ for all $x,y\in U$.
We define the left uniformity $\,\ds\mT$ in a similar way, using
$\,\ds{V} = \{ (x,y)\in G^2 %\mbox{ (or } B^2\mbox{) }
| x^{-1}y\in V\}$ instead.
%We will never use the left completion because
The inverse map
is an isomorphism of uniform spaces $\mbox{Inv}:(G,\mT\ds)\rightarrow(G,\,\ds\mT)$
imposing an isomorphism between the right and left completions. 

The completion $\widehat{G}$ is always a monoid, although the multiplication
$\mbox{Mult}:(G,\mT\ds)\times(G,\mT\ds)\rightarrow(G,\mT\ds)$
is not uniformly continuous in general.
It is a monoid because $\mbox{Mult}(\mF,\mG)$ is a Cauchy filter
for two Cauchy filters $\mF,\;\mG$ on $(G,\mT\ds)$ \cite[Prop. III.3.4(6)]{kn::Bou}. 
On the other hand, the completion is not necessarily a group:  
$\mbox{Inv}:(G,\mT\ds)\rightarrow(G,\,\ds\mT)$
is uniformly continuous but we have no information about
uniform continuity of 
$\mbox{Inv}:(G,\mT\ds)\rightarrow(G,\mT\ds)$.
The latter uniform continuity is a sufficient
(but not necessary) condition for $\widehat{G}$ to be a group.
A necessary and sufficient condition is the following: 
if $\mF$ is a Cauchy filter on $(G,\mT\ds)$,
then $\mbox{Inv}(\mF)$ is a Cauchy filter on $(G,\mT\ds)$
\cite[Th. III.3.4(1)]{kn::Bou}. 

For reader's convenience we sketch an example of a group $G$ with non-a-group $\widehat{G}$
following a hint in Bourbaki \cite[Exercise X.3(16)]{kn::Bou}.
Let $G$ be the group of auto-homeomorphisms of $[0,1]$
with the topology of uniform convergence. It suffices to exhibit a uniformly convergent sequence
$f_m$ of homeomorphisms such that the sequence of inverses $f_m^{-1}$ is not uniformly convergent.
The following sequence fits the bill:
$$
f_m (x)=
\begin{cases} 
  x^m & \mbox{ if } \  x\leq \frac{1}{2} \\
  (2-2^{1-m})x+(2^{1-m}-1) & \mbox{ if } \  x\geq \frac{1}{2} 
   \end{cases}
$$

Now suppose that $G$ admits a BN-pair $(B,N)$.
The key question is whether the completion $\widehat{G}$
admits a BN-pair. Let $\overline{B}$ be the closure of $B$ in $\widehat{G}$.
It is a moot point that
$\overline{B}$ is isomorphic to the completion of $B$ in the restriction uniformity $\mT\ds |_B$ \cite[Cor. II.3.9(1)]{kn::Bou}. 
%from $(G,\mT\ds)$ \cite[Cor. II.3.9(1)]{kn::Bou}. 
A candidate BN-pair on $\widehat{G}$ is $(\overline{B},N)$
but it does not work in general.
Let $\mG$ be a simple split group scheme, $G=\mG (\bF [z,z^{-1}])$
its points over Laurent polynomials over a finite field,
$N\leq G$ the group of monomial matrices,
$I_{-}= [\mG(\bF [z^{-1}])\xrightarrow{z^{-1}\mapsto 0} \mG (\bF)]^{-1} (B)$
its negative Iwahori.
The pair $(I_{-},N)$ is a BN-pair on $G$ but
$(\overline{I_{-}},N)=(I_{-},N)$ is not a BN-pair on the positive completion
$\widehat{G}=\mG (\bF((z))\,)$: the countable groups $I$ and $N$ cannot generate uncountable $\widehat{G}$.
A reader can see that the condition 3 of Theorem~\ref{complete_1} fails for the negative Iwahori. On the other hand, if $\bF$ is finite, all conditions of Theorem~\ref{complete_1} holds for the positive Iwahori 
$I_{+}= [\mG(\bF [z])\xrightarrow{z\mapsto 0} \mG (\bF)]^{-1} (B)$
so that the theorem yields the standard BN-pair on $\widehat{G}$.

%Unfortunately we cannot answer the key question
%in full generality.
%\marginpar{Inna, the answer must be negative.
%Should we try to think of a silly counterexample?}
Nevertheless,  we can prove the following partial affirmative
answer, 
sufficiently general for the study of Kac-Moody groups.
Let us explain some notations before stating the theorem.
The notations $(B,N)$ and $W=(W,S)$ are standard for the groups with BN-pairs.
The homomorphism $\pi: N\rightarrow W$ is the natural surjection.
For elements $s\in S$, $w\in W$ we choose some liftings
$\dot{s}\in \pi^{-1}(s)$, $\dot{w}\in \pi^{-1}(w)$.
The minimal parabolic $P_s$ is the subgroup generated by $B$ and $\dot{s}$.

\begin{theorem}
\label{complete_1}
Let $G$ be a Hausdorff topological group with a BN-pair $(B,N)$
with the Weyl group $(W,S)$ where $S$ is finite. 
If the following three conditions hold,
then  $(\overline{B},N)$ is
a BN-pair on the completed group $\widehat{G}$.
\begin{enumerate}
\item The completion $\widehat{G}$ is a group.
%\item Each minimal parabolic $P_s$ is split as a semidirect
%  product $P_s=L_s\ltimes U_s$ where $U_s$ is open in $G$ and
%  $L_s$ is a group
%  with rank 1 BN-pair $(B\cap L_s, N\cap L_s)$.
%  {\bf Not yet used in full: finite index of $B$ in $P_s$ is used}
\item The index $|P_s:B|$ is finite for all $s\in S$.
\item $B$ is open in $G$.
%\item If $s,t\in S$ are conjugate in $W$, then
%there exists $g\in G$ such that $L_s=gL_tg^{-1}$ and
%$P_s=gP_tg^{-1}$.
%  {\bf Not yet used}
\end{enumerate}
\end{theorem}
\begin{proof}
%Since $\widehat{G}$ is a group,
We have systems of
subgroups:  
$\fA= (B,N,P_s; s\in S)$ of $G$
and 
$\fB= (\overline{B},N,\overline{P_s}; s\in S)$
of $\widehat{G}$, where $\overline{X}$ is the closure of $X$. 
The system of groups $\fA$ satisfies all conditions
of Tits' Theorem as observed by Tits \cite{kn::T}.
We claim that under
the assumptions of this theorem %allow us to conclude that
the system $\fB$ also satisfies these conditions.
We will verify this claim at the end of the proof.
For the reader's convenience we restate Tits theorem (cf. \cite[Th. 5.1.8]{Kum}):
\begin{thmT}
  Suppose that the system $\fB$ satisfies the following conditions:
%${\bf (P_1)}$--${\bf (P_9)}$:
  \begin{itemize}
\item[${\bf (P_1)}$]
If $s\neq t \in S$, then $\overline{P_s}\cap \overline{P_t} = \overline{B}$.
\item[${\bf (P_2)}$]
The subgroup $\overline{B}\cap N$ is normal in $N$.
\item[${\bf (P_3)}$]
  Given $s\in S$, let $N_s \coloneqq \overline{P_s}\cap N$.
Then  $N_s/(\overline{B}\cap N)$ is of order 2 for all $s\in S$.
\item[${\bf (P_4)}$]
$\overline{P_s} = 
\overline{B} \cup \overline{B}\dot{s} \overline{B}$ for all $s\in S$.
\item[${\bf (P_5)}$]
The pair $(N/(\overline{B}\cap N),S)$ is a Coxeter group.
\item[${\bf (P_6)}$]
  Let $\pi: N \rightarrow W \coloneqq N/(\overline{B}\cap N)$ be the quotient map.
  For any $n=\dot{s}_1\dot{s}_2\cdots \dot{s}_t\in N$ with $s_i\in S$  such that $s_1s_2\cdots s_t$ is a reduced word in $W$,
  the subgroup $\overline{B}(s_1, \ldots , s_{t})$
  (see \eqref{sub_n})  depends only on $w \coloneqq \pi (n) ={s}_1\cdots {s}_t\in W$ 
  and the homomorphism $\overline{\gamma}(\dot{s}_1, \ldots, \dot{s}_{t})$
  (see \eqref{hom_n}) depends only on $n$.
  (This justifies the notation $\overline{B}_w$ and $\overline{\gamma}_n$ from now on.)
\item[${\bf (P_7)}$]
  If $w\in W$, $s\in S$ satisfy $l(ws)=l(w)+1$, then $\overline{B}_w\overline{B}_s=\overline{B}_{ws}$. 
\item[${\bf (P_8)}$]
  If $w\in W$, $s,t\in S$ satisfy $wtw^{-1}=s$ and $l(wt)=l(w)+1$,
  then for all $x\in\pi^{-1}(s)$, $n\in\pi^{-1}(w)$ and $b\in\overline{B}\setminus\overline{B}_t$,
  there exist $y\in b\overline{B}_t\cap \overline{B}_n$ and $y^{\prime},y^{\sharp}\in\overline{B}_n$ such that
  \begin{itemize}
  \item[(a)]
$x^{\prime\, -1}yx^{\prime}=y^{\prime}x^{\prime}y^{\sharp}$ in $\overline{P_t}$ and 
  \item[(b)]
$x \overline{\gamma}_n (y) x^{-1} = \overline{\gamma}_n (y^\prime) x^{-1} \overline{\gamma}_n(y^\sharp)$ in $\overline{P_s}$.
  \end{itemize}
  where $x^{\prime} \coloneqq n^{-1}x^{-1}n \in \pi^{-1} (t)$.
\item[${\bf (P_9)}$]
  $\overline{B}$ is not normal in any $\overline{P_s}$.
\end{itemize}
  Then the canonical map
  $$N \cup (\cup_{s\in S}\overline{P_s}) \longrightarrow \widetilde{G} \coloneqq\underset{\fB}{\ast}\;H$$
  to the amalgam $\widetilde{G}$ is injective.
  The amalgam $\widetilde{G}$ admits a BN-pair $(\overline{B}, N)$ with a set of simple reflections $S$, where we identify the groups $\overline{B}$ and $N$ with their images in $\widetilde{G}$ under the canonical map.

  Furthermore, consider a group $G^{\prime}$ and an injective function
  $$\varphi : N \cup (\cup_{s\in S}\overline{P_s}) \longrightarrow {G}^{\prime}$$
  such that $\varphi|_N$ and all $\varphi|_{\overline{P_s}}$ are group homomorphisms.
  If $G^{\prime}$ is generated by the image of $\varphi$, then  the canonical homomorphism $\varphi^{\sharp}: \widetilde{G}\rightarrow G^{\prime}$ is an isomorphism.
\end{thmT}

%Let $\widetilde{G}= \underset{\fB}{\ast}\;H$ be the amalgam of $\fB$. %the second system.
%By Tits' Theorem \cite{kn::T} (cf. \cite[Th. 5.1.8]{Kum}), the group $\widetilde{G}$ admits a BN-pair $(\overline{B},N)$ and the natural group homomorphism $f:\widetilde{G} \rightarrow \widehat{G}$ is injective.

Tits' Theorem applies to $G^{\prime}=\widehat{G}$, allowing us to conclude that
the group $\widetilde{G}$ admits a BN-pair $(\overline{B},N)$
and the natural group homomorphism $f:\widetilde{G} \rightarrow \widehat{G}$ is injective.

It remains to check 
surjectivity of $f$. The image $f(\widetilde{G})$
contains $N$ and $B$, which generate $G$.
Hence, $f(\widetilde{G})$ is dense in $\widehat{G}$. 
On the other hand, $f(\widetilde{G})$ contains $\overline{B}$,
which is open in $\widehat{G}$ because $B$ is open in $G$.
Thus, $f(\widetilde{G})$ is open in $\widehat{G}$
but it is a subgroup, hence, $f(\widetilde{G})$ is also closed in $\widehat{G}$.
Being closed and dense,  $f(\widetilde{G})$ must be equal to $\widehat{G}$.

It only remains to verify all nine conditions in Tits' Theorem for $\fB$.
%For reader's convenience, we use Kumar's condition labels but restate each condition in the last sentence of its proof that always starts from ``Therefore''.
%Recall that these nine conditions hold in $G$ as shown by Tits \cite{kn::T}.
Our starting point is that these nine conditions hold in $G$ for $\fA$ as shown by Tits \cite{kn::T}. 

${\bf (P_2)}$
We know that $B\cap N$ is normal in $N$.
Clearly, $B\cap N \subseteq \overline{B}\cap N$.
In the opposite direction,
$\overline{B}\cap N = (\overline{B}\cap G)\cap N$.
An element $x\in\overline{B}\cap G$
is a limit of a net
$$
x = \lim_{m\in \bM} b_m \; , \ \ \ b_m \in B, \ \ \bM \mbox{ is an ordinal.}
$$
This limit works in $G$ as well where $B$ is open, hence, closed.
Thus, $x\in B$ and $\overline{B}\cap G=B$.
%Then $gxg^{-1} = \lim_{n} gb_mg^{-1}$ for any $g\in N$,
%since left and right multiplications are continuous.
%Since $B$ is open,
%there exists an ordinal $N$ such that
%$gb_m x^{-1}g^{-1} \in B$
%and consequently $b_mx^{-1}\in g^{-1}Bg$ for all $n>N$.
%Thus, $\overline{B}\cap N = \overline{B\cap N}$ is
Therefore, $\overline{B}\cap N = B\cap N$ is normal in $N$.

${\bf (P_1)}$ Let $s\neq t \in S$. Since
$P_s= B\dot{s}B \cup B$,
%$P_s\cup P_t = B$,
we conclude that
%$\overline{P_s}\cap \overline{P_t} \supseteq \overline{P_s\cup P_t} = \overline{B}$.
$\overline{P_s} = \overline{B\dot{s}B} \cup \overline{B}$.
%=\overline{B}\dot{s}\overline{B} \cup\overline{B}$.
%\cap \overline{P_t} \supseteq \overline{P_s\cup P_t} = \overline{B}$.    
Let us prove now that
$\overline{B\dot{s}B} \cap \overline{B\dot{t}B} =\emptyset$.
An element
$x\in \overline{B\dot{s}B} \cap \overline{B\dot{t}B}$
is a limit of two nets
$$
x = \lim_{m\in \bM} a_m \dot{s} b_m = \lim_{m\in \bM} c_m \dot{t} d_m\; , \ \ \ a_m,b_m,c_m,d_m \in B. 
$$
Since $\overline{B}$ is open there exists an ordinal $\bL<\bM$ such that
$a_m \dot{s} b_m x^{-1} \in \overline{B} \ni x (c_m \dot{t} d_m)^{-1}$
and consequently $a_m \dot{s} b_m d_m^{-1} \dot{t}^{-1} c_m^{-1} \in \overline{B}$
for all $m\geq \bL$.
Clearly $a_m \dot{s} b_m d_m^{-1} \dot{t}^{-1} c_m^{-1} \in G$.
It is shown ${\bf (P_2)}$ that $\overline{B}\cap G=B$, thus,
$a_m \dot{s} b_m d_m^{-1} \dot{t}^{-1} c_m^{-1} \in B$
for all $m\geq \bL$.
On the other hand,  these elements $a_m \dot{s} b_m d_m^{-1} \dot{t}^{-1} c_m^{-1}$
lie in $B\dot{s}B \dot{t}B$ equal
to the double coset $B\dot{s}\dot{t}B$ since
$l(\dot{s}\dot{t})=2=l(\dot{s})+l(\dot{t})$. 
%$\overline{P_s}=\overline{L_s\ltimes U_s} = L_s \overline{U_s}$
%since $U_s$ is open and $L_s$ is consequently discrete.
%Using Bruhat decomposition for the BN-pair on $L_s$,
%$$
%\overline{P_s}=
%[(L_s\cup B)\dot{\cup} (L_s\cup B)\dot{s}(L_s\cup B)] \overline{U_s}
%=
%\o{B}\dot{\cup} (L_s\cup B)\dot{s}\overline{B}.
%$$
Since $B\dot{s}B \dot{t}B\cap B = \emptyset$, 
no such $x$ exists. 
Therefore, $\overline{P_s}\cap \overline{P_t} = \overline{B}$.

${\bf (P_3)}$ The minimal parabolic $P_s$ is a union of cosets of $B$,
hence, open in $G$.
Similarly to the proof in ${\bf (P_2)}$,
$N_s = \overline{P_s}\cap N$ is equal to
$(\overline{P_s}\cap G) \cap N = P_s\cap N$.
Therefore, $N_s/(\overline{B}\cap N)$ is of order 2 for all $s\in S$.

${\bf (P_4)}$
By condition (2), $B$ has finite index in $P_s$. Hence
$B\dot{s} B = XB$ for some finite subset $X\subseteq P_s$.
Observe now that
$$
\overline{B}\dot{s} \overline{B} \subseteq \overline{B\dot{s} B} =
\overline{XB} \stackrel{f}{\supseteq} X \overline{B}
\subseteq \overline{B}\dot{s} \overline{B}. $$
Since $X$ is finite, $X\overline{B}$ is closed and the inclusion $f$ is an equality.
Thus, $\overline{B\dot{s}B}=\overline{B}\dot{s} \overline{B}$.
Therefore, 
$\overline{P_s} = \overline{B}\cup \overline{B\dot{s}B} =
\overline{B} \cup \overline{B}\dot{s} \overline{B}$.
%(Observe that the union is disjoint by ${\bf (P_3)}$.)

${\bf (P_5)}$ We have proved in ${\bf (P_2)}$ that $\overline{B}\cap N = B\cap N$.
Therefore, $(N/\overline{B}\cap N,S) =(W ,S)$ is a Coxeter group.

${\bf (P_6)}$ Let us first observe that if $A,K\leq G$ are open subgroups,
then $\overline{A}\cap\overline{K}=\overline{A\cap K}$ in $\widehat{G}$.
Indeed, the inclusion $\supseteq$ is obvious. To prove the inclusion $\subseteq$ 
consider $x\in\overline{A}\cap\overline{K}$.
%An element
%$x\in \overline{B\dot{s}B} \cap \overline{B\dot{t}B}$
It is a limit of two nets
$$
x = \lim_{m\in \bM} a_m  = \lim_{m\in \bM} k_m\; , \ \ \ a_m\in A,\ k_m \in K. 
$$
Then the net $a_mk_m^{-1}=(a_mx^{-1})(xk_m^{-1})$ converges to $1\in G$.
Since $A\cap K$ is open there exists an ordinal $\bL<\bM$ such that
$a_mk_m^{-1}\in A\cap K$, and consequently
$a_m,k_m\in A\cap K$
for all $m\geq \bL$.
Thus, $x\in\overline{A\cap K}$.

The following subgroups are defined recursively
for a reduced word \newline
$s_1s_2\cdots s_t\in W$
and its fixed lift $n=\dot{s}_1\dot{s}_2\cdots \dot{s}_t\in N$
$$
%B(1)\coloneqq B, \; \overline{B}(1)\coloneqq\overline{B}, \; 
B(s_1, \ldots , s_{i})\coloneqq B\cap \dot{s}_i^{-1} B(s_1,  \ldots , s_{i-1})\dot{s}_i,
$$
\begin{equation} \label{sub_n}
\overline{B}(s_1,  \ldots , s_{i})\coloneqq \overline{B}\cap \dot{s}_i^{-1}
\overline{B}(s_1, \ldots ,s_{i-1})\dot{s}_i.
\end{equation} 
The aforementioned observation implies that
$\overline{B}(s_1, \ldots, s_{t})= \overline{B(s_1, \ldots ,s_{t})}$.
Consequently, the homomorphism
\begin{equation} \label{hom_n}
\overline{\gamma}(\dot{s}_1, \ldots, \dot{s}_{t}) :
\overline{B}(s_1, \ldots ,s_{t}) \rightarrow \overline{B}, \ \
x\mapsto \dot{s}_1 \cdots \dot{s}_{t} x \dot{s}_t^{-1} \cdots \dot{s}_{1}^{-1}
\end{equation}
is uniquely determined by its restriction
$\gamma(\dot{s}_1, \ldots, \dot{s}_{t}) :
B(s_1, \ldots ,s_{t}) \rightarrow B$. 
%The properties of BN-pair
Property ${\bf (P_6)}$ hold for $\fA$.
This means that 
$B(s_1, \ldots ,s_{t})$ 
depends only
on the element $w=s_1 \cdots s_t \in W$, not the word or the choice
of the liftings $\dot{s}_i\in N$.
It also means that
$\gamma(\dot{s}_1,  \ldots, \dot{s}_{t})$ depends only
on the element $n=\dot{s}_1 \cdots \dot{s}_t \in N$.
Therefore, the subgroup
$\overline{B}(s_1, \ldots ,s_{t})$ and the homomorphism
$\overline{\gamma}(\dot{s}_1,  \ldots, \dot{s}_{t})$ depend only
on $w \in W$ and $n\in N$ correspondingly.
(We denote these $B_w, \overline{B}_w, \gamma_n, \overline{\gamma}_n$.)

${\bf (P_7)}$ We begin by proving that all subgroups $B_w$, $w\in W$
are commensurable. We proceed by induction on the length $l(w)$
to show that $B_w$ has finite index in $B$.

If $l(w)=1$,
then $w=s$ for some $s\in S$.
Since $xB_s \mapsto x\dot{s}B$ is an embedding of quotient sets
$B/B_s \hookrightarrow P_s/B$, we conclude that 
$|B:B_s| \leq |P_s:B| < \infty$ by assumption (2).

Suppose the case of $l(w)=m-1$ is settled.
Consider $w\in W$, $s\in S$  with $l(ws)=m$.
%There exists $s\in S$ such that $l(ws)=n-1$ so that
Then
$
B_{ws} = B \cap \dot{s}^{-1}B_{w}\,\dot{s}
$. 
Hence,
$$
|B:B_{ws}| = |B:B_s||(B\cap \dot{s}^{-1}B\dot{s}):(B\cap \dot{s}^{-1}B_{w}\dot{s})|
\leq
|B:B_s||B:B_{w}| <\infty
$$
since by induction assumption $|B:B_{w}|<\infty$.

Property ${\bf (P_7)}$ for $\fA$ ensures that
%The properties of BN-pair ensure that
$B_{w}B_s=B$ if $l(ws)=l(w)+1$ as before.
Hence,
$$
\overline{B} = \overline{B_{w}B_s}
\supseteq
\overline{B_{w}}\; \overline{B_s}
=
\overline{B}_{w}\, \overline{B}_s 
$$
because $\overline{XY}\supseteq\overline{X}\;\overline{Y}$
for all subsets $X,Y$ and
$\overline{B_w} = \overline{B}_w$ for all $w\in W$
as shown in  ${\bf (P_6)}$.  
Since $B_w$ and $B_s$ are commensurable, $B_w B_s = X B_s$
for a finite subset $X\subseteq B_w$.
Then
$$
\overline{B_{w}B_s}=\overline{XB_s} = X \overline{B_s}
\subseteq \overline{B}_{w}\,\overline{B}_s
$$
since $X \overline{B_s}$ is closed as a finite union of closed cosets of $\overline{B_s}$.
Therefore,
$
\overline{B}_{w}\, \overline{B}_s
= \overline{B}$.

${\bf (P_8)}$
Let $s,t\in S$, $w\in W$ such that
$wtw^{-1}=s$ and $l(wt)=l(w)+1$. 
%If $\pi:N\rightarrow W$ is the natural surjection,
Let us fix arbitrary $x\in\pi^{-1}(s)$, $n\in\pi^{-1}(w)$
and define $x^\prime \coloneqq n^{-1}x^{-1}n\in\pi^{-1}(t)$.
Now pick any $b\in \overline{B}\setminus\overline{B}_t$.
As shown in ${\bf (P_7)}$, $B=X B_t$ for a finite set $X$.
Without loss of generality, $1\in X$ and
$B\setminus B_t=(X\setminus\{1\}) B_t$.
By the argument as above ($X\overline{B}_t$ is closed etc.),
$\overline{B}\setminus \overline{B}_t=(X\setminus\{1\}) \overline{B}_t$.
Hence, $b=b_1b_2$ for some $b_1\in X\setminus\{1\}$
and $b_2\in \overline{B}_t$. This brings property
${\bf (P_8)}$ down to the system $\fA$ %the group $G$ from $\widetilde{G}$
where
we know it \cite{kn::T}. %\marginpar{Why? Check Tits}
Therefore, there exist elements
$y\in b_1B_t\cap B_w \subseteq b\overline{B}_t \cap \overline{B}_w$
and
$y^\prime,y^{\sharp}\in B_w \subseteq \overline{B}_w$
satisfying
$x^{\prime\, -1}yx^\prime =y^\prime x^\prime y^{\sharp} \in
P_t \subseteq \overline{P_t}$
and
$x \gamma_n (y) x^{-1} = \gamma_n (y^\prime) x^{-1} \gamma_n(y^\sharp)
\in P_s \subseteq \overline{P_s}$.

${\bf (P_9)}$
Recall that $\overline{B}\cap G =B$ and $\overline{P_s}\cap G = P_s$
as shown in ${\bf (P_2)}$.
If $\overline{B}$ were normal in $\overline{P_s}$, then
$B$ would be normal in $P_s$, contradicting ${\bf (P_9)}$ for $\fA$. 
%which is not possible in a group with a BN-pair. 
Therefore, $\overline{B}$ is not normal in $\overline{P_s}$.
\end{proof}

A shortcoming of Theorem~\ref{complete_1}
is that it requires the group $\widehat{G}$ to exist first.
It would be useful to tweak %modify
the theorem to enable construction of new groups. The next theorem
addresses this issue. If $\mT$ is a topology on a group $B$,
we denote $\mT_1 \coloneqq \{ A\in \mT \,\mid\, 1\in A\}$. 
\begin{theorem}
  \label{complete_2}
Let $G$ be a group with a BN-pair $(B,N)$
with the Weyl group $(W,S)$ where $S$ is finite. 
Suppose further that a topology $\mT$ on $B$ is given
such that
the four conditions ($1$)--($4$)
%or the five conditions ($1$)--($4$) and ($5^\dagger$)
hold.
\begin{enumerate}
%\begin{mylist}
\item $(B,\mT)$ is a topological group.
\item The completion $\widehat{B}$ is a group.
\item $\mT_1$ is a basis at $1$ of topology on each minimal parabolic
  $P_s$, $s\in S$ that defines a structure of topological group on $P_s$.
\item The index $|P_s:B|$ is finite for each $s\in S$. 
%  is split as a sIf $\overline{U_s}$ is a$\widehat{B}$ be the
%then  $(\overline{B},N)$ is a BN-pair on the completed group $\widehat{G}$.
%  subgroups $\fA=\{\overline{B},N,\overline{P_s}; s\in S\}$.
%Let $\widetilde{G}= \underset{\fA}{\ast}\;H$ be its amalgam.
%emidirect
%  product $P_s=L_s\ltimes U_s$ where $U_s$ is an open subgroup of $B$.
%\item{($3$)}   $L_s$ is a finite(??) group  with rank 1 BN-pair $(B\cap L_s, N\cap L_s)$.
%\item{($4$)} If $s,t\in S$ are conjugate in $W$, then
%there exists $g\in G$ such that $L_s=gL_tg^{-1}$ and
%$P_s=gP_tg^{-1}$.
%\item{($5$)} Action of each $L_s$ on $U_s$ is continuous.
%\item{($5^\dagger$)} A basis $\mJ$ of topology on $B$ consists of subgroups $H$ such that
%$H\cup U_i$ have finite index in $U_i$ for each $i$. 
%\end{mylist}
\end{enumerate}
%If $\overline{U_s}$ is the closure  of $U_s$ in $\widehat{B}$, 
%then $\fA=\{\widehat{B},N,L_s\ltimes\overline{U_s}; s\in S\}$
%is a system of subgroups (cf.  \cite[Def. 5.1.6]{Kum})
%whose amalgam $\underset{\fA}{\ast}\;H$ is a topological group
%with BN-pair $(\widehat{B},N)$.
Under these conditions the following statements hold:  
%If $\overline{U_s}$ is the closure  of $U_s$ in $\widehat{B}$, 
%then $\fA=\{\widehat{B},N,L_s\ltimes\overline{U_s}; s\in S\}$
%is a system of subgroups (cf.  \cite[Def. 5.1.6]{Kum})
%whose amalgam $\underset{\fA}{\ast}\;H$ is a topological group
%with BN-pair $(\widehat{B},N)$.
\begin{enumerate}
%\begin{mylist}
\item[(a)]
  $\mT_1$ is a basis at $1$ of topology on $G$
  that defines a structure of topological group on $G$.
\item[(b)]  The completion $\widehat{G}$ is a group and $\widehat{B}=\overline{B}^{\subseteq\widehat{G}}$.
\item[(c)] The completion $\widehat{G}$ is isomorphic
  to the amalgam $\underset{\fB}{\ast}\;H$
  where \newline $\fB=\{\overline{B},N,\overline{P_s}; s\in S\}$.
\item[(d)] The pair
$(\overline{B},N)$ is a BN-pair on the completed group $\widehat{G}$.
\end{enumerate}
%\end{mylist}
\end{theorem}
\begin{proof}
  {\bf (a)} We already know that $\mT_1$ is a filter of neighbourhoods of $1$
  in a topological group $B$. To verify (a) it suffices to show
  that for all $g\in G$, $A\in \mT_1$
  it holds
  %there exist
  %$V\in\mT_1$ such
  that $gAg^{-1}\in \mT_1$
  \cite[Prop. III.1.2(1)]{kn::Bou}. By (3) we know this property for all $g\in P_s$.
  Since $G$ is generated by all $P_s$, we conclude the proof.

  {\bf (b)} Denote the aforementioned topology on $G$ by $\mT_G$.
  We need to show that the monoid $\widehat{(G,\mT_G)}$ is a group.
  Consider  $x\in \widehat{(G,\mT_G)}$ and a convergent net $x_m\longrightarrow x$,
  $m\in \bM$, $x_m \in G$. Since $B\in \mT_1$, $B$ is open in $(G,\mT_G)$.
  The net $x_m$ is Cauchy, so
  there exists an ordinal $\bL$ such that $x_mx_\bL^{-1}\in B$ for all $m\geq \bL$.
  Let
  $$
  y_m=\begin{cases} 
  1 & \mbox{ if } \  m< \bL,   \\
  %(x_\bLx_m^{-1})^{-1}=
  x_mx_\bL^{-1} & \mbox{ if } \  m\geq \bL.  
  \end{cases}
  $$
  Since $y_my_l^{-1}=(x_mx_\bL^{-1})(x_lx_\bL^{-1})^{-1}=x_mx_l^{-1}$,
  the net $y_m$ is a Cauchy net in $B$. Let $y=\lim y_m\in\widehat{B}$. The
  inverse $y^{-1}\in\widehat{B}$ exists because $\widehat{B}$ is a group. Then
  $$
  (x_\bL^{-1} y^{-1})\cdot x = x_\bL^{-1} \cdot \lim_m y_m^{-1} \cdot \lim x_m =
  x_\bL^{-1} \cdot \lim_m y_m^{-1}x_m = 
  x_\bL^{-1} \cdot \lim_m x_\bL = 1.  
  $$
  Similarly, $x\cdot (x_\bL^{-1} y^{-1})$. Thus, we have found the inverse
  $x^{-1}= x_\bL^{-1} y^{-1})\in\widehat{G}$ so that $\widehat{B}$ is a group.
Coincidence of the completion and the closure is standard. 
  
  {\bf (c+d)} These follow immediately from Theorem~\ref{complete_1}.
\end{proof}

Suppose that a group $G$ admits two topological group structures
$(G,\mS)$ and $(G,\mT)$ such that $\mS\subseteq \mT$.
Then the identity map $\Id : (G,\mT) \rightarrow (G,\mS)$
is a homomorphism of topological groups
that admits a unique extension
$\widehat{\Id} : \widehat{(G,\mT)} \rightarrow \widehat{(G,\mS)}$
\cite[Prop. III.3.4(8)]{kn::Bou}. 
This extension $\widehat{\Id}$ may or may not be injective in general.
Ditto for surjective
\cite[Exercise III.3(12)]{kn::Bou}.
However, we can give nice criteria for surjectivity and injectivity
for the topologies we are interested in.

\begin{cor}
  \label{injective}
  Consider a group $G$ with a BN-pair $(B,N)$
  that admits two topological group structures
  $(G,\mS)$ and $(G,\mT)$ such that $\mS\subseteq \mT$.
  Suppose
%$(G,\mT)$ satisfies the conditions of Theorem~\ref{complete_1}
%or Theorem~\ref{complete_2}.
$B\in \mS$. 
%  If $\;\widehat{\Id_B} : \widehat{(B,\mT_B)} \rightarrow \widehat{(B,\mS_B)}$
%  is injective, then 
%  $\widehat{\Id} : \widehat{(G,\mT)} \rightarrow \widehat{(G,\mS)}$
%  is injective.
  Then the kernel of $\widehat{\Id} : \widehat{(G,\mT)} \rightarrow \widehat{(G,\mS)}$
  is equal to the kernel of 
  $\;\widehat{\Id_B} : \widehat{(B,\mT_B)} \rightarrow \widehat{(B,\mS_B)}$.
\end{cor}
\begin{proof}
  Clearly, $\ker (\widehat{\Id})\supseteq\ker (\widehat{\Id_B})$.
  In the opposite direction,  
  consider  $x\in\ker (\widehat{\Id})$. This element is a limit
  of a Cauchy net $x_m\in G,\; m\in \bM$ in $\mT$ such that
  $x_m \longrightarrow 1$ in $\mS$. 
  Since $B\in\mS$ there exists an ordinal $\bL<\bM$ such that
  $x_m\in B$ for all $m\geq \bL$.
  Thus, $x\in \widehat{(B,\mT_B)}$ and  $x\in \ker (\widehat{\Id_B})$.
\end{proof}

\begin{cor}
  \label{surjective}
  Consider a group $G$ with a BN-pair $(B,N)$
  that admits two topological group structures
  $(G,\mS)$ and $(G,\mT)$ such that $\mS\subseteq \mT$.
  Suppose $(G,\mS)$ satisfies the conditions of Theorem~\ref{complete_1}
  or Theorem~\ref{complete_2}.
  If $\;\widehat{\Id_B} : \widehat{(B,\mT_B)} \rightarrow \widehat{(B,\mS_B)}$
  is surjective, then 
  $\widehat{\Id} : \widehat{(G,\mT)} \rightarrow \widehat{(G,\mS)}$
  is surjective.
\end{cor}
\begin{proof}
  This holds because $\widehat{(G,\mS)}$ is generated by $N$
  and $\overline{B}$.
\end{proof}

Surjectivity of $\widehat{\Id}$ has a very interesting consequence as pointed out to us by Guy Rousseau.
Note  that the map
$\widehat{\Id} : \widehat{(G,\mT)} \rightarrow \widehat{(G,\mS)}$
defines an injective map of Tits buildings 
$\mBT (\widehat{(G,\mT)})\rightarrow \mBT(\widehat{(G,\mS)})$.
Surjectivity of $\widehat{\Id}$ implies that this map of Tits buildings is bijective.

We finish this section with a convenient corollary
of Theorem~\ref{complete_2} whose proof is straightforward.
%Should we discuss the following??
%
%Suppose further that
%each $L_i$ is finite
%and the basis $\mJ$ consists of subgroups $H$ such that
%$H\cup U_i$ have finite index in $U_i$ for each $i$. 
%Pro-p-completion

\begin{cor}
  \label{complete_3}
Let $G$ be a group with a BN-pair $(B,N)$
with the Weyl group $(W,S)$ where $S$ is finite. 
Suppose further that a system $\fS$ of subgroups of $B$ is given
such that
the following three conditions hold.
\begin{enumerate}
\item $\fS$ forms a topology basis at $1$ of 
a topological group $(B,\mT)$.  
\item Each minimal parabolic $P_s$
    is split as a semidirect
    product $P_s=L_s\ltimes U_s$ where $L_s$ is a finite group and $U_s$
    is a subgroup of $B$.
\item   $L_s$ acts continuously on $(U_s,\mT_{U_s})$.
\end{enumerate}
Then the four conclusions of Theorem~\ref{complete_2} hold.
%whose proof is straightforward.
%If $\overline{U_s}$ is the closure  of $U_s$ in $\widehat{B}$, 
%then  
%$\fA=\{\widehat{B},N,L_s\ltimes\overline{U_s}; s\in S\}$
%is a system of subgroups (cf.  \cite[Def. 5.1.6]{Kum})
%whose amalgam $\underset{\fA}{\ast}\;H$ is a topological group
%with BN-pair
%$(\widehat{B},N)$.
\end{cor}

\section{Completions of Kac-Moody Groups}

Let $A=(a_{ij})_{n\times n}$ be a generalised Cartan matrix,
$\mD=(I, A, \mX,\mY,\Pi,\Pi^\vee )$  a root datum of type $A$.
Recall that this means 
\begin{itemize}
\item $I=\{1, 2, ... , n\}$,
\item $\mY$ is a free finitely generated abelian group,
\item $\mX=\mY^\ast =\hom (\mY,\bZ)$  is its dual  group,
\item $\Pi=\{\alpha_1, \ldots \alpha_n\}$ is a set of simple roots,
where $\alpha_i \in \mX$,
\item $\Pi^\vee =\{\alpha^\vee_1, \ldots \alpha^\vee_n\}$ is a set of  
 simple coroots,
where $\alpha^\vee_i \in \mY$,
\item for all $i,j\in I$,  
$\alpha_j (\alpha^\vee_i ) = a_{ij}$.
\end{itemize}

Recall that $\mD$ is {\em simply connected} if $\Pi^\vee$ is a basis
of $\mY$. 
%(see Section~\ref{sc6} for an example of a simply connected
%affine group),
We call $A$  \linebreak 
$2$-spherical if  
 for each $J\subseteq I$ with $|J|=2$, the submatrix 
$A_J{\coloneqq}(a_{ij})_{i,j\in J}$ is a  Cartan matrix of finite type. 
Let $\Delta^{re}$ be the set of real roots. 
Recall that  $\Delta^{re}=W(\Pi)$ where $W$ is the Weyl group.
Note that
$\Delta^{re}=\Delta^{re}_+\cup \Delta^{re}_-$.  

 Let $\mathbb{F}=\bF_q$  be a finite field of $q=p^a$ elements ($a\geq 1$ and $p$ a prime). 
Tits \cite{T2,T3} gives a definition of a Kac-Moody group  $G{\coloneqq}G_\mD(\bF)$, 
%From Tits' definition an explicit presentation of these groups  was derived by Carter,  a presentation by the explicit set of generators and relations \`{a} la Steinberg (cf. \cite{kn::Car} p. 224). The presentation depends on the field and root datum, so the resulting group can be denoted $G_\mD(\bF)$.
%If $A$ is not of finite type (i.e., $A$ is not  a  Cartan matrix from classical Lie theory), the standard presentation is infinite:
 %the number of generators and the number of relations are both infinite.
which is generated by the torus $T=\hom (\mX,\bF^\times )$ and root subgroups
$X_{\alpha}\cong \mathbb{F}^+$,  $\alpha\in\Delta^{re}$.
For all $i\in I$, set 
$$M_i\coloneqq\langle X_{\alpha_i}\cup X_{-\alpha_i}\rangle.$$

The group $G$ admits a $BN$-pair $(B,N)$ 
%(in fact two $BN$-pairs $(B_+,N)$ and $(B_-,N)$) 
where $B=U\rtimes T$ with $U{\coloneqq}\langle X_{\alpha}\mid \alpha\in\Delta^{re}_+\rangle$ and $N{\coloneqq}N_G(T)$, the normaliser of $T$.  
If $\mD$ is simply connected, $T\cong (\bF^\times)^n$. 
Moreover, the Coxeter group $(N/T,S=\{s_i, i\in I \})$ and  $(W, S)$ are isomorphic.
We  denote by $P_i{\coloneqq}P_{s_i}$, $i\in I$, a minimal parabolic subgroup of $G$. 
It is known \cite[6.2]{kn::RCsimp} that $P_i=U_i\rtimes L_i$ where $L_i{\coloneqq}M_iT$ and 
$U_i=U\cap s_i Us_i^{-1}$.
In particular,
$|P_i:B|$ is finite for all $i\in I$.

Consider the set of subgroups $\mF$ of $B$ where
$$\mF=\{ A \,\mid\, A \trianglelefteq U \leq B, \; |U:A|=p^a\ \mbox{for some}\ a\in\bN\}.$$
Elements of $\mathcal{F}$ form a basis at $1$ of a topology on $B$. 
Then the completion $\widehat{B}$ of $B$ with respect to this topology is a group. 
Since $|P_s:U|<\infty$, conditions (1)--(4) of Theorem~\ref{complete_2} are satisfied,  thus its conclusions  hold.
In particular, $\widehat{G}$ is a topological group with an open subgroup $\widehat{B}$  and $(\widehat{B}, N)$ is a $BN$-pair of $\widehat{G}$.

Since $U$ is a residually finite-$p$ group \cite[Remark after Th.~4.1]{kn::E2} and
\linebreak $|B:U|\mid (q-1)^n$, 
our completion $\widehat{B}$ is equal to $\widehat{U}\rtimes T$ where $\widehat{U}$ is the full pro-$p$ completion of $U$.

Let us recall other known topological Kac-Moody groups.
They are the Mathieu-Rousseau group  $G^{ma+}$, the Carbone-Garland group $G^{c\lambda}$  and  the Caprace-R\'emy-Ronan group $G^{crr}$. Each of them contains
a quotient $G^{\dagger}\coloneqq G/Z$ by a central subgroup $Z$,
which depends on the completion and could be trivial. 
In fact, $G^{\dagger}$ is always dense in $G^{c\lambda}$ and $G^{crr}$. Let $\overline{G}^+$ be the closure of $G^{\dagger}$ in $G^{ma+}$. Rousseau \cite[6.10]{kn::Rou} investigates  whether $\overline{G}^+$ equals $G^{ma+}$ and show that this happens when 
 $p>max\{|a_{ij}|, i\neq j\}$  \cite[6.11]{kn::Rou}. Rousseau and later Marquis give examples when it does not happen \cite{kn::Rou}, \cite{Mar2}.

There are two further known completions of $G$ where the closure $\overline{U}$ is compact totally disconnected \cite{RWe}. {\em The Belyaev group}  $G^{b}$ is the ``largest'' such completion.   {\em The Schlichting group}  $G^{s}$ is the ``smallest'' such completion. Our completion admits a characterisation similar to {the Belyaev group}: $\widehat{G}$ is the ``largest'' completion where the closure $\overline{U}$ is a pro-$p$-group. 

Let $\overline{U}$ and $\overline{B}$ be the closures of $U$ and $B$ correspondingly  in either of the topological groups $\overline{G}^+$, $G^{c\lambda}$ or $G^{crr}$.
The group homomorphism $\widehat{U}\twoheadrightarrow\overline{U}$
extends to
 $\widehat{B}\twoheadrightarrow\overline{B}$ and  $\widehat{G}\twoheadrightarrow\overline{G}$ (cf. Section 6.3 of \cite{kn::Rou}).
Using this and the universal properties of the Belyaev and Schlichting completions, we have  open continuous homomorphisms \cite[6.3]{kn::Rou}:
\begin{equation} \label{homomorphisms}
G^{b}\twoheadrightarrow\widehat{G}\twoheadrightarrow \overline{G}^+\twoheadrightarrow G^{c\lambda}\twoheadrightarrow G^{crr}\xrightarrow{\cong} G^{s}.
\end{equation}

It is known that for $\overline{G}\in \{\overline{G}^+,G^{c\lambda},G^{crr}\}$, $\overline{G}/Z'(\overline{G})$ is topologically simple (where $Z'(\overline{G})=\bigcap_{g\in\overline{G}} g\overline{B}g^{-1}$ ).
What about our new group $\widehat{G}$?

Recall the following criterion of Bourbaki \cite{kn::Bou2}.

\begin{prop}
\label{prop::BBK}
Let $(G,B,N,S)$ be a Tits system with Weyl group $W = N/(B \cap N)$. Let $U$ be a subgroup of $B$. 
We set $Z'(G) = \bigcap_{g \in G} gBg^{-1}$. Assume that $G$ is a topological group topologically generated by the conjugates of $U$ in $G$. Assume further $B$ a closed subgroup of $G$, and  the following conditions hold:
\begin{itemize}
\item[{\rm (1)}] We have $U \triangleleft B$ and $B = UT$ where $T = B \cap N$.
\item[{\rm (2)}] For any proper normal  closed subgroup $V \triangleleft U$, we have $[U/V,U/V] \subsetneq U/V$.
% BR : added "normal subgroup"
\item[{\rm (3)}]  Subgroup $[G,G]$ is dense in $G$.
\item[{\rm (4)}] The Coxeter system $(W,S)$ is irreducible.
\end{itemize}

\noindent Then for any normal closed subgroup $K$ in $G$,  $K \leq Z'(G)$. In particular, $G/Z'(G)$ is topologically simple.
\end{prop}

We now prove the following statement.

\begin{theorem}
\label{theorem::top_simple}
If $A$ is an irreducible generalised Cartan matrix, then
$\widehat{G}/Z'(\widehat{G})$ is topologically simple.
\end{theorem}
\begin{proof}
If $V$ is a closed normal subgroup of $\widehat{U}$, then 
$[\widehat{U}/V, \widehat{U}/V]\neq \widehat{U}/V$ as shown in 
\cite[4.4]{kn::E}.
Now Proposition~\ref{prop::BBK} finishes the proof.
\end{proof}

 There is a similar criterion for the abstract simplicity  \cite{kn::Bou2}. 

\begin{prop}
\label{th::BBK2}
Let $(G,B,N,S)$ be a Tits system with Weyl group $W = N/(B \cap N)$.
Let $U$ be a subgroup of $B$ such that $G$ is generated by the conjugates of $U$. 
%We set $Z'(G) = \bigcap_{g \in G} gBg^{-1}$.
Assume that the following holds.
\begin{itemize}
\item[{\rm (1)}] We have $U \triangleleft B$ and $B = UT$ where $T = B \cap N$.
\item[{\rm (2)}] For any proper normal subgroup $V \triangleleft U$, we have $[U/V,U/V] \subsetneq U/V$.
\item[{\rm (3)}] We have $G=[G,G]$.
\item[{\rm (4)}] The Coxeter system $(W,S)$ is irreducible.
\end{itemize}

\noindent Then for any normal subgroup $K$ in $G$,  $K \leq Z'(G)$. In particular, $G/Z'(G)$ is abstractly simple.
\end{prop}

This allows us to prove the following statement.
 
 \begin{theorem}
   \label{theorem::alg_simple}
   Suppose $q\geq 4$. 
   If $A$ is irreducible and $2$-spherical,
   then $\widehat{G}/Z'(\widehat{G})$ is abstractly simple, and
 there are natural isomorphisms
 $$\widehat{G}/Z^\prime(\widehat{G})\xrightarrow{\cong}\overline{G}^+/Z^\prime(\overline{G}^+)\xrightarrow{\cong} G^{c\lambda}/Z^\prime(G^{c\lambda})\xrightarrow{\cong} G^{crr}/Z^\prime( G^{crr}).$$
\end{theorem}
\begin{proof}
  Let us first show that $\widehat{G}/Z^\prime(\widehat{G})$ is abstractly simple. To do that it suffices to check the conditions of Proposition~\ref{th::BBK2} for the Tits system $(\widehat{G},\widehat{B},N,S)$.

  By construction of $\widehat{G}$, condition (1) holds because it holds in $(G,{B},N,S)$.

 Abramenko proves that  for $q\geq 4$, $U$ is finitely generated if and only if $A$ is $2$-spherical \cite{kn::A}.
 Thus,  $\widehat{U}$ is topologically finitely generated.
 By  \cite[Lemma 4.4]{kn::E}, condition (2) of Proposition~\ref{th::BBK2} holds for $\widehat{U}$ and any proper normal subgroup $V$ of $\widehat{U}$.

 Moreover, $[\widehat{G},\widehat{G}]\geq [G,G][\widehat{U},\widehat{U}]$.
 Since $q\geq 4$, for every $\alpha\in\Delta^{re}$, the subgroup $M_{\alpha}{\coloneqq}\langle X_{\alpha}, X_{-\alpha}\rangle$ is perfect (in fact, it is $\PSL_2 (\bF)$ or $\SL_2 (\bF)$), and thus $[M_{\alpha}, M_{\alpha}]=M_{\alpha}\supseteq X_{\alpha}$. Hence, $G=[G,G]$.
  Now the argument of Carbone, Ershov and Ritter \cite[4.3(b)]{kn::E} shows that
$[\widehat{U},\widehat{U}]$ is an open subgroup of $\widehat{G}$, and so 
$\widehat{G}=[\widehat{G}, \widehat{G}]$. 
 
Finally, condition (4) holds since $A$ is irreducible. Therefore, $\widehat{G}/Z^\prime(\widehat{G})$ is abstractly simple.

Observe that the homomorphisms~\eqref{homomorphisms} yield open surective homomorphisms
$$
\widehat{G}/Z^\prime(\widehat{G})\twoheadrightarrow
\overline{G}^+/Z^\prime(\overline{G}^+)\twoheadrightarrow
G^{c\lambda}/Z^\prime(G^{c\lambda})\twoheadrightarrow
G^{crr}/Z^\prime( G^{crr})$$
that are isomorphisms of abstract groups due to simplicity  of $\widehat{G}/Z^\prime(\widehat{G})$. 
They are isomorphisms of topological groups because they are open.
\end{proof}

\section{Congruence Kernel}
We finish the paper with some observations
on the structure of $Z'(\widehat{G})$.
To facilitate our discussion we use the following notation for arbitrary groups $K\leq H$:
\begin{itemize}
\item $\widehat{H}$ -- the completion of $H$ in the pro-$p$
  topology on $H$ or its canonical (such as  $U$) subgroup,
\item $\widetilde{H}$ -- the completion of $H$ in some other topology,
\item $\overline{K}^{\subseteq H}$ (or simply $\overline{K}$) --
  the closure of $K$ in $H$,
\item $C(H,K)\coloneqq\cap_{g\in H}gKg^{-1}$ --
the normal core of $K$ in $H$.
%  ``{\em the congruence kernel}'' for a subgroup $K$ of $H$.
\end{itemize}
The group $Z'(\widehat{G})$
contains two commuting subgroups:
the centre (before completion) $Z(G)$ and 
the normal core %congruence kernel
$C(\widehat{G},\widehat{U})$.
In fact, as $\widehat{U}\cong\overline{U}^{\subseteq \widehat{G}}$
is a Sylow pro-$p$ subgroup of $\widehat{G}$,  $C(\widehat{G},\widehat{U})=C(\widehat{G}, V)$ for any Sylow pro-$p$ subgroup $V$ of $\widehat{G}$.
Therefore, we may use the notation 
$C(\widehat{G})$ instead of $C(\widehat{G},\widehat{U})$. 
Sometimes it is  convenient to use
the full notation $C(\widehat{G},\widehat{U})$.
We will use both notations depending on circumstances.

Following the argument of  Rousseau \cite[Prop. 6.4]{kn::Rou}, 
 we can prove that 
$$Z'(\widehat{G})= Z(G)\times C(\widehat{G}).$$
We can compute the centre $Z(G)$ from
the Cartan matrix but
we see no efficient way of 
computing the normal core $C(\widehat{G})$.
Observe that in the  Caprace-R\'emy-Ronan completion, $Z'(G^{crr})=Z(G)$.
Hence, Theorem~\ref{theorem::alg_simple} implies that $C(\widehat{G})=\ker(\phi)$ where $\phi:\widehat{G}\rightarrow G^{crr}$ is the natural continuous open surjective homomorphism. The kernels of the natural maps between two different completions of the same groups are commonly known as {\em congruence kernels}, the term used later in the paper.
Can we describe $C(\widehat{G})$ explicitly?

Let $\mP$ be the collection of all normal index $p^n$,
$n\in \bN$, subgroups of $U$
so that
$$
\widehat{U} \cong
\{ (x_H H) \,\mid\,
x_H\in U,\;
H\in \mP,\;
x_HH = x_{H^\prime} H
\mbox{ for }
H\geq H^\prime \}
\leq \prod_{H\in\mP} U/H\, .
$$
Let us examine the action of $U$ on the Tits building
$\mBT(\widehat{G})$.
Let $U_n$ be
the pointwise stabiliser of the ball of radius $n$ around the simplex
$\widehat{B}$ in $\mBT(\widehat{G})$.
Then 
$$\mP^0 \coloneqq \{ H\in \mP \,\mid\, \exists n \,:\, H \geq U_n \}$$
is a basis of topology on $U$.  We can describe the completion
of $U$ in this topology as
$$
U^{crr}\cong
\{ (x_H H) \,\mid\,
x_H\in U,\;
H\in \mP^0,\;
x_HH = x_{H^\prime} H
\mbox{ for }
H\geq H^\prime \}
\leq \prod_{H\in\mP^0} U/H\, .
$$
Clearly, $C (G^{crr},U^{crr})=1$ because it consists of those  elements
$(x_HH)$ that act trivially on $\mBT(\widehat{G})$.
This forces $x_H \in U_n$ for all $n$
and $(x_HH)=1$. The natural map $\widehat{U}\rightarrow U^{crr}$
is the projection whose kernel is exactly $C(\widehat{G},\widehat{U})$
that we can describe now as
$$
C (\widehat{G})
=
\{ (x_H H) \,\mid\,
x_H\in H^\star,\;
H\in \mP,\;
x_HH = x_{H^\prime} H
\mbox{ for }
H\geq H^\prime \}
\leq \prod_{H\in\mP} H^\star /H
$$
where 
$H^\star \coloneqq \cap_{H\subseteq K \in \mP^0} K$.  
This description tells us that one of the three following statements holds:
\begin{enumerate}
\item $\mP\setminus\mP^0$ is finite. Then $C (\widehat{G},\widehat{U})$ is a finite group.
\item $\mP\setminus\mP^0$ is infinite but
  $\{H^\star\,\mid\, H\in \mP\setminus\mP^0\}$ is finite.
  Then $C (\widehat{G},\widehat{U})$ is a finitely generated pro-$p$ group.
\item $\{H^\star\,\mid\, H\in \mP\setminus\mP^0\}$ is infinite.
  Then $C (\widehat{G},\widehat{U})$ may be an infinitely generated pro-$p$ 
  group.
\end{enumerate}

A natural question to address is whether $C (\widehat{G})$
is central. We can do it under some strong assumptions.
%Remark that the following statement holds.
\begin{lemma}
  If $A$ is irreducible of indefinite type and $q{\geq}n>2$,
  then at least one of the following statements holds:
  \begin{enumerate}
  \item $C (\widehat{G})$
  is not a finitely generated pro-$p$ group,  
  \item $C (\widehat{G})\leq Z (\widehat{G})$.
  \end{enumerate}
  In particular, if $C (\widehat{G})$ is finite, then $C (\widehat{G})\leq Z (\widehat{G})$.
\end{lemma}
\begin{proof}
If $A$ is irreducible of indefinite type and $q{\geq}n>2$, then
$G/Z(G)$ is a simple non-linear group as shown by Caprace and R\'{e}my \cite{kn::CaR}.

Let us assume that $C{\coloneqq}C (\widehat{G})$ is a finitely generated pro-$p$ group.
%Since $C$ is a finitely generated pro-$p$ group,
In this case the Frattini quotient $C /\Phi(C )$ is a finite elementary abelian $p$-group. Since $\Phi(C )\triangleleft\widehat{G}$, it follows immediately that $\widehat{G}/Z'(\widehat{G})$ acts on $C/\Phi(C )$.
Now $\widehat{G}/Z'(\widehat{G})$ contains a dense subgroup isomorphic to $G/Z(G)$.
This subgroup is simple non-linear, hence, it must act trivially on the finite group $C /\Phi(C)$.
Since the subgroup is dense, the whole $\widehat{G}/Z'(\widehat{G})$ acts trivially.
Since the action is given by conjugation $g\cdot (c \Phi (C)) = (gcg^{-1}) \Phi (C)$,
we can say that $\widehat{G}/Z'(\widehat{G})$ centralises $C /\Phi(C)$.

Now let $T$ be the torus of $G$, defined at the start of Section~2.
Then 
 $[T,C/\Phi(C)]=1$. 
Let $C_i{\coloneqq}\Phi_i(C)$, the $i$-th Frattini subgroup.
Since $C$ is finitely generated,
$C/C_i$ is a finite $p$-group, $\Phi(C/C_i)=\Phi(C)/C_i$ and
$\{C_i, i\in\mathbb{N}\}$ is a fundamental system of open neighbourhoods of $1$ in $C$
\cite[2.8.13]{kn::RZ}.
It follows that $T$ acts on $C/C_i$ and centralises $(C/C_i)/\Phi(C/C_i)$.
A theorem of Burnside states that a $p^\prime$-automorphism of a $p$-group $P$, inducing the identity automorphism on $P/\Phi (P)$, is the identity itself \cite[5.1.4]{kn::G}. It follows that $[T, C/C_i]=1$.
 Since $\{C_i, i\in\mathbb{N}\}$ is a fundamental system in $C$, it follows that $[T,C]=1$.
 Obviously $[T^g,C]=1$ for all $g\in\widehat{G}$. Therefore, 
 $[\langle T^g, g\in\widehat{G}\rangle, C]=1$.
 Since $Z(G)$ is a subgroup of $T$  \cite[Cor. 5.14]{kn::RCsimp},
 $G  = \langle T^g, g\in{G}\rangle \leq \langle T^g, g\in\widehat{G}\rangle$. The result now follows.
\end{proof}

In some cases we can describe 
$C (\widehat{G})$ fully.
%For instance,
%if $A$ is of finite type, then $G_\mD(\bF)$
%is finite and $C(\widehat{G},\widehat{U})=1$.
%A similar result holds for the most of affine groups:
\begin{prop}
  \label{affine}
Suppose that the generalised Cartan matrix
$A=(a_{ij})_{n\times n}$ 
is irreducible, of untwisted affine type
and $n\geq 3$. Then  $C (\widehat{G})=1$. In particular,
$$
\widehat{G}\cong G^{ma+}\cong G^{c\lambda}\cong G^{crr}
$$
$$
\mbox{and } \qquad 
\widehat{\mG(\mathbb{F}_q[t,t^{-1}] \, )}
  \cong \mG(\mathbb{F}_q((t)) \, ) \, .
  $$
\end{prop}
\begin{proof} %Let us start with the untwisted case.
  The root datum $\mD$ changes only Cartan subgroup
  and has no effect on $U$ or $C (\widehat{G})=C (\widehat{G},\widehat{U})$.
  Thus we may choose $\mD$ so that $G\cong \mG(\mathbb{F}_q[t, t^{-1}])$ for the corresponding  Chevalley group scheme $\mG$. Now Lemma 7 of \cite{kn::CRL} gives us that $\widehat{U}$ is a Sylow pro-$p$ subgroup of $\mG(\mathbb{F}_q((t))$. This implies that $\widehat{G}\cong \mG(\mathbb{F}_q((t)))$ which gives the desired result.
\end{proof}
We expect Proposition~\ref{affine} to hold
for a twisted affine $A$ as well. As pointed out by the referee, it would
be interesting to establish whether
$
\widehat{G}\cong G^{crr}
$
implies that $G$ is of affine type. 
For instance, the isomorphism fails in rank 2 as shown in the next proposition.
%We cannot describe $C (\widehat{G})$ fully
%but we can determine which of the three
%scenaria actually takes place:
\begin{prop}
  \label{non-affine}
Suppose that the generalised Cartan matrix
$A=(a_{ij})_{2\times 2}$ 
is not of finite type and $p>\max\{|a_{12}|, |a_{21}|\}$. 
Then     $C (\widehat{G})$ is 
an infinitely-generated pro-$p$-group
and $\{H^\star\,\mid\, H\in \mP\setminus\mP^0\}$ is infinite.
\end{prop}
\begin{proof}
For such $A$, Morita \cite[3(6)]{kn::Mo}
gives  the description of $U$ in $G_{\mD}(\mathbb{F})$ for a field $\mathbb{F}$  of characteristic $0$. His description extends to the case $\mathbb{F}=\mathbb{F}_q$, as an interested reader can verify.
If $\min\{|a_{12}|, |a_{21}|\}\geq 2$, then  $U=U_1*U_2$, where $U_i\cong\mathbb{F}_q[t]$ for $i=1,2$.
If  $\min\{|a_{12}|, |a_{21}|\}=1$, then $U=U_1*U_2$ and each $U_i$ is a metabelian infinitely generated group.  

Consider $\widehat{U}$. By \cite[9.1.1]{kn::RZ}, $\widehat{U}=\widehat{U}_1\amalg \widehat{U}_2$, the free pro-$p$ product of the pro-$p$ groups $\widehat{U}_1$ and $\widehat{U}_2$. Since each $\widehat{U}_i$ is an infinitely-generated pro-$p$ group,   \cite[9.1.15]{kn::RZ} implies that $\widehat{U}$ is infinitely-generated.

On the other hand, as $p>\max\{|a_{12}|, |a_{21}|\}$,  the results of \cite[2.2 and 2.4]{kn::CR} give us that $U^{crr}$ is a finitely generated pro-$p$ group. The proposition follows immediately.
\end{proof}

\begin{cor}
\label{Rank2Minor}
  Let $A=(a_{ij})_{n\times n}$ be an irreducible generalised Cartan matrix whose Dynkin diagram contains an infinite edge,
i.e., there exists $1\leq i\neq j\leq n$ with $a_{ij}a_{ji}\geq 4$. Suppose that $p>\max\{|a_{ij}|, i\neq j\}$.
Then     $C (\widehat{G},\widehat{U})$ is an infinitely-generated pro-$p$-group
and $\{H^\star\,\mid\, H\in \mP^0\}$ is infinite.
\end{cor}
\begin{proof}
Let $P$ be a parabolic subgroup of $G$ whose Levi complement corresponds to the subdiagram of the Dynkin diagram $\Delta$ of $G$ based on $\alpha_i$ and $\alpha_j$.
Then $P=U_P\rtimes L$ where 
$$L=\langle X_{\alpha}, T\mid \alpha\in \Delta^{re}\cap \Sp_{\mathbb{Z}}\{\alpha_i, \alpha_j\} \rangle$$ is a Levi complement of $P$ and $U_P=\cap_{g\in P} U^g$ \cite[6.2.2]{kn::R}.
Hence, $U=U_P\rtimes U_L$ where $U_L=U\cap L$.
It follows that $\widehat{U_L}\leq \widehat{U}$. Moreover, the natural isomorphism $U/U_P \xrightarrow{\cong} {U_L}$
yields an exact sequence
\begin{equation}\label{seqA1}
1 \longrightarrow
U_P
\xrightarrow{a}
U
\xrightarrow{b}  
U_L
\rightarrow 1.
\end{equation}
Since pro-$p$-completion is a right exact functor,
 the sequence
\begin{equation}\label{seqA2}
\widehat{U_P}
\xrightarrow{\widehat{a}}
\widehat{U}
\xrightarrow{\widehat{b}}
\widehat{U_L}
\rightarrow 1
\end{equation}
is exact as well. 
Therefore, $\widehat{U_L}$ is a homomorphic image of $\widehat{U}$. Since $\widehat{U_L}$ is an infinitely-generated pro-$p$ group, so is $\widehat{U}$.

As $p>\max\{|a_{ij}|, i\neq j\}$, 
the results of \cite{kn::CR} imply that $U^{crr}$ is a finitely generated pro-$p$ group.
Note that $C (\widehat{G},\widehat{U})$ is the kernel of the homomorphism
$\widehat{U}\rightarrow U^{crr}$. 
This finishes the proof. \end{proof}

It is possible to relate the calculations of
the congruence kernel of a Levi factor and of the unipotent
radical of a parabolic.
Let $J\subseteq I$ and $P{\coloneqq}P_J$
 a parabolic in $G_\mD(\bF)$
 with the unipotent radical $U_P$ and
 a Levi complement $L=\langle X_{\alpha}, T \mid \alpha\in \Delta^{re}\cap \Sp_{\mathbb{Z}}(J)\rangle$. Then $P=U_P\rtimes L$ \cite[6.2.2]{kn::R}
 and we have a natural isomorphism
$\widehat{L} \xrightarrow{\cong} \overline{L}^{\subseteq\widehat{G}}=\overline{L}^{\subseteq\overline{P}}$ where  $\overline{P}\coloneqq \overline{P}^{\subseteq\widehat{G}}$. 
%The closure in $\widehat{G}$ of $U_P$ is $\overline{U_P}$. 
 Let $U_L{\coloneqq}U\cap L$.
 Then $U_L$ is the unipotent radical of a Borel subgroup of $L$.
Two ``parabolic'' congruence kernels
%$$
%C (\widehat{G},\overline{U_P})\coloneqq\cap_{g\in\widehat{G}}\, g\overline{U_P}g^{-1}
%\ \mbox{ and } \ 
%C (\widehat{L},\widehat{U_L})\coloneqq\cap_{g\in\widehat{L}}\, g\widehat{U_L}g^{-1}
%$$
``approximate'' $C (\widehat{G},\widehat{U})$:
\begin{theorem}
  \label{approximate}
There exists an exact sequence of topological groups
\begin{equation}\label{seq1}
1 \longrightarrow
C (\widehat{G},\overline{U_P})
\longrightarrow
C (\widehat{G},\widehat{U})
\longrightarrow
%\prod_{g\in G} g C(\widehat{L},\widehat{U_L})g^{-1}
C(\widehat{L},\widehat{U_L})\rightarrow 1 \; . 
\end{equation}

%where the number of components in the product $\prod C(\widehat{L},\widehat{U_L})$
%is the index $|W:W_L|$ of the corresponding Coxeter subgroup.
%
%Moreover, the projection
%$
%C(\widehat{G},\widehat{U})
%\longrightarrow
%C (\widehat{L},\widehat{U_L})
%$
%is surjective.
Moreover, $C(\widehat{L}, \widehat{U_L})$ is a subgroup of $C(\widehat{G}, \widehat{U})$. In particular,
if $C (\widehat{L})\neq 1$ 
then
$C(\widehat{G})\neq 1$.
\end{theorem}

\begin{proof}
  Notice that $\widehat{L}$ is a topological Kac-Moody group on its own letting us talk about $C(\widehat{L})=C(\widehat{L}, \widehat{U_L})$.
  
%The natural isomorphism $U/U_P \xrightarrow{\cong} {U_L}$
%yields an exact sequence
%$$
%1 \longrightarrow
%U_P
%\xrightarrow{a}
%U
%\xrightarrow{b}  
%U_L
%\rightarrow 1.
%$$
%Since pro-$p$-completion is a right exact functor,
%the sequence
%$$
%\widehat{U_P}
%\xrightarrow{\widehat{a}}
%\widehat{U}
%\xrightarrow{\widehat{b}}
%\widehat{U_L}
%\rightarrow 1
%$$
%is exact as well. 
  Let us examine the exact sequence~(\ref{seqA2}) in the proof of
  Corollary~\ref{Rank2Minor}. 
The image $\widehat{a}(\widehat{U_P})$ is a closed subgroup
containing $a(U_P)$.
Since $U_P$ is dense in $\widehat{U_P}$, 
$a(U_P)$ is dense in $\widehat{a}(\widehat{U_P})$.
This yields another exact sequence
\begin{equation}\label{seq2}
1 \longrightarrow
\overline{U_P}
\xrightarrow{\overline{a}}
\widehat{U}
\xrightarrow{\widehat{b}}
\widehat{U_L}
\rightarrow 1.
\end{equation}
The same argument applied to the semidirect decomposition $P=U_P\rtimes L$
gives an exact sequence with $\widehat{c}|_{\widehat{U}}=\widehat{b}$: 
\begin{equation}\label{seq2a}
1 \longrightarrow
\overline{U_P}
\xrightarrow{\overline{a}}
\overline{P}
\xrightarrow{\widehat{c}}
\widehat{L}
\rightarrow 1
\end{equation}
where the closure $\overline{P}=\overline{P}^{\subseteq\widehat{G}}$
is the completion of $P$ in the uniformity induced from $\widehat{G}$.
Loosely speaking,
both $\overline{P}$ and $\widehat{G}$ are obtained by pro-$p$-completion of $U$.
Both
$\overline{U_P}$
and
$\widehat{U}$
are subgroups of $\overline{P}$.
The map $\overline{a}$ is the inclusion of subgroups.
Conjugating them by all $g\in \overline{P}$
and then intersecting yields the inclusion
$\overline{a}:C (\overline{P},\overline{U_P})
\hookrightarrow
C (\overline{P},\widehat{U})$.
Moreover,
the sequence~(\ref{seq2}) restricts to a new sequence
\begin{equation}\label{seq3}
1 \longrightarrow
C (\overline{P},\overline{U_P})
\xrightarrow{\overline{a}}
C (\overline{P},\widehat{U})
\xrightarrow{\widehat{b}}
C(\widehat{L},\widehat{U_L})\; .
\end{equation}
%To prove this we examine the $\overline{P}$-action on $\widehat{L}$. 
Observe that $p\cdot y {\coloneqq} \widehat{c}(p)y \widehat{c}(p)^{-1}$,
$p\in\overline{P}$, $y\in\widehat{L}$, gives
a $\overline{P}$-action on $\widehat{L}$. 
It follows that
$\widehat{b} (C (\overline{P},\widehat{U}))
\subseteq C(\widehat{L},\widehat{U_L})$,
so that the sequence~(\ref{seq3}) is well-defined. 

We can conjugate 
$C (\overline{P},\overline{U_P})$
and
$C (\overline{P},\widehat{U})$
by all $g\in\widehat{G}$ and
intersect further.
This yields a subsequence of the sequence~\eqref{seq3}:
\begin{equation}\label{seq4}
1 \longrightarrow
C (\widehat{G},\overline{U_P})
\xrightarrow{\overline{a}}
C (\widehat{G},\widehat{U})
\xrightarrow{\widehat{b}}
C(\widehat{L},\widehat{U_L})\; . 
\end{equation}
This sequence is precisely the sequence~\eqref{seq1} in the statement of the theorem. It remains to establish surjectivity of $\widehat{b}$ in the sequence~\eqref{seq4}.

Consider the restriction of the continuous
homomorphism $\phi:\widehat{G}\rightarrow G^{crr}$ to $\widehat{L}$.
Clearly, $\phi(\widehat{L})=L^{crr}$.
In particular, $\phi(C(\widehat{L},\widehat{U_L}))\subseteq L^{crr}$.
%%%On the other hand, $C(\widehat{L},\widehat{U_L})$
%%%fixes the link in $\mBT (\widehat{G})=\mBT(G^{crr})$ defined by $L$.
We have an $\widehat{L}$-equivariant map 
$$
\eta: \mBT (\widehat{L})_m = \widehat{L}/(\overline{B}\cap\widehat{L})
\rightarrow
\mBT (\widehat{G})_m = \widehat{G}/\overline{B}, \ \
g(\overline{B}\cap\widehat{L})
\mapsto g\overline{B}, 
$$
where by $\mBT (\widehat{G})_m$ and $\mBT (\widehat{L})_m$ we denote the set of simplices of maximal dimension in the corresponding Tits buildings. 
As a subset of $\mBT (\widehat{G})=\mBT(G^{crr})$, the image of $\eta$ consists of those simplices that have $P$ as a face because, corestricted to its image, $\eta$ can be identified with the natural map $\widehat{L}/(\overline{B}\cap\widehat{L}) \rightarrow \overline{P}/\overline{B}$

Since  $C(\widehat{L},\widehat{U_L})$ acts trivially on  $\mBT (\widehat{L})$,
it follows that $C(\widehat{L},\widehat{U_L})$
fixes the image of $\eta$. Since the stabiliser of an individual simplex is a Borel subgroup, the fixator of all these simplices is $C(\overline{P}, \overline{B})$. 
It follows from \cite[Th 6.3]{kn::RCsimp} 
that $C(\overline{P}, \overline{B})$ is equal to $U_P^{crr}T^\prime$, where $T^\prime$ is a subgroup of a torus in $\overline{B}$. Therefore, $\phi(C(\widehat{L},\widehat{U_L}))\subseteq  U_P^{crr}T^\prime$.

Now $C(\widehat{L},\widehat{U_L})$ and $U_P^{crr}$ are pro-$p$-groups, while $T^\prime$ is a finite $p^{\prime}$-group, i.e., a group of order coprime to $p$. 
Thus, $\phi(C(\widehat{L},\widehat{U_L}))\subseteq U_P^{crr}$. 
Furthermore, $\phi(C(\widehat{L},\widehat{U_L}))\subseteq L^{crr}\cap U_P^{crr}=1$. 
It follows that $C(\widehat{L},\widehat{U_L})$ is contained
in the kernel of $\phi$ that is equal to
$Z'(\widehat{G})= Z(G)\times C(\widehat{G},\widehat{U})$.
Since $Z(G)$ is a subgroup of the torus \cite[Cor. 5.14]{kn::RCsimp},
it is a finite $p^{\prime}$-group, while $C(\widehat{G},\widehat{U})$ is a pro-$p$-group.
It follows that $C(\widehat{L},\widehat{U_L})$ is contained in $C(\widehat{G},\widehat{U})$.
This inclusion splits the sequence~(\ref{seq4}) proving surjectivity of $\widehat{b}$.
\end{proof}

\end{document}